\documentclass[11pt]{amsart}

\usepackage{amsmath, amsthm, amsfonts, amssymb}
\newtheorem{prop}{Proposition}
\newtheorem{theorem}{Theorem}
\newtheorem{coroll}{Corollary}
\newtheorem{lemma}{Lemma}

\newtheorem{remark}{Remark}

\newcommand{\ig}{[\hspace{-1.5pt}[}
\newcommand{\id}{]\hspace{-1.5pt}]}

\newcommand{\C}{\mathbb{C}}

\newcommand{\un}{\mathbf{1}}
\renewcommand{\P}{\mathbb{P}}

\newcommand{\R}{\mathbb{R}}

\renewcommand{\H}{\mathcal{H}}

\newcommand{\T}{\mathbb{T}}

\author{Jean B\'{e}rard,  Jean-Baptiste Gou\'{e}r\'{e}}
\address[Jean B\'{e}rard]{\noindent Institut Camille Jordan, UMR CNRS 5208, 43, boulevard du 11 novembre
1918, Villeurbanne, F-69622, France; universit\'{e} de Lyon, Lyon, F-69003, France; 
universit\'{e} Lyon 1, Lyon, F-69003, France
\newline
e-mail:  \rm \texttt{jean.berard@univ-lyon1.fr}}
\address[Jean-Baptiste Gou\'{e}r\'{e}]{\noindent Laboratoire MAPMO -
  UMR 6628, Universit\'e d'Or\-l\'eans, B.P. 6759, 45067 Orl\'eans Cedex
  2, France. 
\newline 
E-mail: \rm
  \texttt{Jean-Baptiste.Gouere@univ-orleans.fr}.}

\date{}

\thanks{The authors would like to thank J. Quastel for suggesting that the approach developed in \cite{MueMytQua, MueMytQua2} might be exploited to re-derive
the asymptotic behavior of the survival probability of the branching random walk.}

\title[Survival probability of the branching random walk]{Survival probability of the branching random walk killed below a linear boundary}

\begin{document}

\begin{abstract}
We give an alternative proof of a result by N. Gantert, Y. Hu and Z. Shi on the asymptotic behavior of the survival probability of the branching random walk killed below a linear boundary, in the special case of deterministic binary branching and bounded random walk steps. Connections with the Brunet-Derrida theory of stochastic fronts are discussed. 
\end{abstract}

\maketitle

\section{Introduction}

%
%

Consider a real-valued branching random walk, with deterministic binary branching and i.i.d. random steps. One possible way of constructing such a branching random walk is to start with a rooted infinite binary tree $\T$ and 
 a collection $(\zeta(e))_{e}$ of  i.i.d. copies of a random variable $\zeta$, where the index $e$ runs over the edges of $\T$. Given $x \in \R$, the value of the branching random walk at every vertex $u$ of $\T$ is then defined as 
$$S(u) := x + \sum_{e \in \ig root ,  u  \id } \zeta(e),$$
where  $\ig root , u \id$ denotes the set of edges of $\T$ lying between $u$ and the root. With this definition, for each infinite ray $root=:x_0,x_1,\ldots$ in $\T$, the sequence $S(x_i)_{i=0,1,\cdots}$ formed by the successive values of the random walk along the ray, is a real-valued random walk started at $S(root)=x$. In the sequel, we use the notation $\P_x$ to reflect the choice of $x$ as the initial value, so that $\P_0$ is used for the case where the random walk starts at zero.

For $v \in \R$, define $A_{\infty}(v)$ to be the event that there exists an infinite ray \mbox{$root=:x_0,x_1,\ldots$} in $\T$ such that $S(x_i) \geq v i$ for all $ i \geq 0$.
A dynamical interpretation of this event is as follows. Assume that the tree $\T$ depicts the genealogical structure of a population whose individuals are identified with the vertices, so that the population starts at generation zero with a single ancestor represented by the root. For each vertex $u \in \T$, $S(u)$ is interpreted at the location on the real line of the individual represented by $u$. Now assume that, in every generation $n \geq 0$, any individual whose location is strictly below $v n $ is killed, and removed from the population together with all its progeny. In this setting, the event  $A_{\infty}(v)$ corresponds to the fact that the population survives the killing process forever , i.e. that the $n-$th generation is non-empty for every $n \geq 0$.

It turns our that, under suitable regularity assumptions on the distribution of $\zeta$, there exists a critical value $v^*$ such that $\P_0(A_{\infty}(v))>0$ for all $v<v^*$, while $\P_0( A_{\infty}(v) )=0$ for all $v \geq v^*$ (see e.g. \cite{BigLubShwWei}). In this paper, we are interested in how fast $\P_0(A_{\infty}(v))$ goes to zero when $v<v^*$ goes to $v^*$.  To state a precise result, we first introduce some notations and assumptions on the distribution of $\zeta$. Our first and rather drastic assumption is that $\zeta$ has bounded support, or, in other words, that there exists a bounded interval $[\zeta_-, \zeta_+]$ such that 
\begin{equation} P( \zeta \in  [\zeta_-, \zeta_+] ) = 1.\end{equation}
Due to this assumption, the log-Laplace transform of $\zeta$,
$$\Lambda(t) := \log  E(\exp(t \zeta)),$$ is well-defined and finite for all $t \in \R$, 
and is $C^{\infty}$ as a function of the parameter $t$. 
Our next (and last) assumption on the distribution of $\zeta$ is that\footnote{The meaning of this assumption is discussed in more detail in the Appendix.} there exists a $t^* \in ]0,\infty[$ such that 
\begin{equation}\label{e:lambda-1}\Lambda(t^*)-t^* \Lambda'(t^*) = -\log(2).\end{equation} 
 With these assumptions, the critical value $v^*$ is given by
\begin{equation}\label{e:lambda-2}v^*=\Lambda'(t^*).\end{equation} 
and a recent result by N. Gantert, Y. Hu and Z. Shi \cite{GanHuShi} gives the following description of the asymptotic behavior of the survival probability as $v$ goes to $v^*$:
\begin{theorem}\label{t:le-theoreme}
 For $v<v^*$, one has the following asymptotic behavior as $v \to v^*$:
$$ \log \P_0(A_{\infty}(v)  ) \sim -  \pi  \sqrt{  \frac{\Lambda''(t^*) t^*} {2 (v^*-v)}}.$$ 
\end{theorem}
In fact, the theorem quoted above is less general than the one proved in \cite{GanHuShi}, since we have assumed deterministic binary branching and bounded random walk steps, whereas \cite{GanHuShi}  allows for a more general branching mechanism, and possibly unbounded steps provided that the Laplace transform is finite in a neighborhood of zero. In this paper, we shall prove  the following slightly  improved version of Theorem \ref{t:le-theoreme}.
\begin{theorem}\label{t:le-theoreme-mieux}
 For $v<v^*$, one has the following asymptotic behavior as $v \to v^*$:
\begin{equation}\label{e:enonce} \log \P_0(A_{\infty}(v)  ) = -  \pi  \sqrt{  \frac{\Lambda''(t^*) t^*} {2 (v^*-v)}} + O \left( \log(v^*-v)    \right).\end{equation} 
\end{theorem}
The main novelty in the present paper lies in the method of proof, which is completely different from that developed in \cite{GanHuShi}.  Indeed, the proof given in \cite{GanHuShi} is based, among other things, on a first-second moment argument, using a change-of-measure technique combined with refined "small  deviations" estimates for random walk paths, exploiting some ideas developed in \cite{Kes} in the context of branching Brownian motion. On the other hand, our proof relies on the characterization of the survival probability of the branching random walk as the solution of a non-linear convolution equation.  This idea was in fact used by B. Derrida and D. Simon in \cite{DerSim1, DerSim2} to derive a heuristic justification of Theorems \ref{t:le-theoreme} and \ref{t:le-theoreme-mieux}, treating the corresponding equation by the non-rigorous methods developed by Brunet and Derrida to study stochastic front propagation models. The rigorous treatment given here was inspired by the work of C. Mueller, L. Mytnik and J. Quastel \cite{MueMytQua, MueMytQua2}, who deal with a continuous-time 
version of the non-linear convolution equation, as an intermediate step in their proof of the Brunet-Derrida velocity shift for the stochastic F-KPP equation. Both our argument and the one given in
\cite{MueMytQua, MueMytQua2} use the idea of comparing the solutions of the original non-linear equation to solutions of suitably adjusted linear approximations of it. An important difference, however, is that 
the equation appearing in  \cite{MueMytQua, MueMytQua2} is a second-order non-linear o.d.e., for which specific techniques (such as phase-plane analysis) can be applied, while such tools are not available in our discrete-time setting, so that we had to find a different way of implementing the comparison idea.

The rest of the paper is organized as follows. 
In Section \ref{s:equations}, we introduce the non-linear convolution equation characterizing the survival probability, 
and show how, given super- and sub-solutions of this equation, one can obtain upper and lower bounds on the survival probability. Section \ref{s:lineaire} is devoted to the study of a linearized version of the convolution equation, for which explicit solutions are available 
when $v$ is close to $v^*$. In Section \ref{s:building} we explain how the explicit solutions of the linearized equation derived in the previous section, can be used to build super- and sub-solutions to the original non-linear convolution equation. Finally, Section \ref{s:conclusion} puts together the arguments needed to prove Theorem \ref{t:le-theoreme-mieux}. Section \ref{s:BD} then discusses the 
connection of Theorem \ref{t:le-theoreme-mieux} and its proof with the Brunet-Derrida theory of stochastic fronts and related rigorous mathematical results. In the appendix, we discuss the meaning of the assumption that  there exists a $t^*$ such that (\ref{e:lambda-1}) holds, and the asymptotic behavior of the survival probability when this assumption is not met.

\section{Equations characterizing the survival probability}\label{s:equations}

In this section, we explain how to the survival probability can be characterized by non-linear convolution equations, and how  super- and sub-solutions to these equations provide control upon this probability.

\subsection{Statement of the results}

Throughout this section, $v$ denotes a real number such that 
$$v<v^*.$$
Let $\psi    \   :   \     [0,1] \to [0,1]$ be defined by 
$$\psi(s)=2s-s^2.$$
Let $\H$ denote the following set of maps:
$$\H  :=    \{      h    \        :        \        \R \to [0,1],       \           h  \mbox{ non-decreasing, 
 $h \equiv 0$ on $]-\infty, 0[$ } \},$$
  and, for all $h \in \H$, define $T(h) \in \H$ by\footnote{The fact that $T(h) \in \H$ is a straightforward consequence of the fact that $h \in \H$ and that $\psi$ is non-decreasing.} 
$$\left\{ \begin{array}{l}    T(h)(x) := \psi ( E(h(x+\zeta-v))),   \    x \geq 0  \\
          T(h)(x) := 0 ,   \        x<0        \end{array}  \right.$$

For all $n \geq 0$, let $A_n(v)$ denote the event that there exists a ray of length $n$ $root=:x_0,x_1,\ldots,x_n$ in $\T$ such that $S(x_i) \geq v i$ for all $ i \in \ig 0,n \id$.
Finally, for all $x \in \R$, let
$$q_n(x) :=  \P_x(A_n(v)),      \        q_{\infty}(x) := \P_x(A_{\infty}(v) ).$$
Note that, for obvious reasons\footnote{For instance by an immediate coupling argument.}, $q_n$ and $q_{\infty}$ belong to $\H$, and $q_0 = \un_{[0,+\infty)}$.
Given $h_1, h_2 \in \H$, we say that $h_1 \leq h_2$ when $h_1(x) \leq h_2(x)$ for all $x \in \R$.

The following proposition gives the non-linear convolution equation satisfied by $q_n$ and $q_{\infty}$, on which our analysis is based.
\begin{prop}\label{p:iter}
For all $n \geq 0$, $q_{n+1} = T(q_n)$, and $T(q_{\infty})=q_{\infty}$.
\end{prop}
\begin{proof}
Analysis of the first step performed by the walk.
\end{proof}
A key property  is the fact that being a non-trivial fixed point of $T$ uniquely characterizes $q_{\infty}$ among the elements of $\H$, 
as stated in the following proposition.
\begin{prop}\label{p:charact-q}
Any $r \in \H$ such that $r \not\equiv 0$ and $T(r) = r$ is such that $r=q_{\infty}$.
\end{prop}
Our strategy for estimating $q_{\infty}(x)$ is based on the following two comparison results.
\begin{prop}\label{p:convergence-super}
Assume that $h \in \H$ is such that $h(0)>0$ and $T(h) \leq h$.
Then $q \leq h$.  
\end{prop}
\begin{prop}\label{p:convergence-sub}
Assume that $h \in \H$ is such that
$T(h) \geq h$.
Then $q \geq h$.  
\end{prop}
 
 \subsection{Proofs}

Let us first record the following simple but crucial monotonicity property of $T$.
\begin{prop}\label{p:monotonie}
If $h_1, h_2 \in \H$ satisfy  $h_1 \leq h_2$, then $T(h_1)\leq T(h_2)$.
\end{prop}
\begin{proof}
Immediate.
\end{proof}

As a first useful consequence of Proposition \ref{p:monotonie}, we can prove Proposition \ref{p:convergence-sub}.
\begin{proof}[Proof of Proposition \ref{p:convergence-sub}]
Since $T(h) \geq h$, we can iteratively apply Proposition \ref{p:monotonie} to prove that 
for all $n \geq 0$, 
$T^{n+1}(h) \geq T^n(h),$ whence the inequality
$T^n(h) \geq h.$
On the other hand, since $h \in \H$, we have that $h \leq \un_{[0,+\infty[}$, whence 
the inequality $T^n(h) \leq T^n(\un_{[0,+\infty[})$. Since $T^n(\un_{[0,+\infty[}) = q_n$ by Proposition \ref{p:iter}, we deduce that 
$q_n \geq h$.
By dominated convergence, $\lim_{n \to +\infty} q_n(x) = q_{\infty}(x)$, so we finally deduce that $h \leq q_{\infty}$.
\end{proof}

We now collect some elementary lemmas that are used in the subsequent proofs.
\begin{lemma}\label{l:depasse-pas}
One has that $P(\zeta \leq v^*)<1$.
\end{lemma} 
\begin{proof}
By assumption, there exists $t^*>0$ such that $\Lambda'(t^*)=v^*$.
But $\Lambda'(t^*) = \frac{ E(\zeta e^{t^* \zeta})   }{E(e^{t^* \zeta})}$, so that 
$E(\zeta e^{t^* \zeta}) =  E(v^* e^{t^* \zeta})$.
If moreover $P(\zeta \leq v^*)=1$, we deduce from the previous identity that $P(\zeta=v^*)=1$, 
so that $\Lambda(t^*)=t^*v^*$, which contradicts the assumption that 
$\Lambda(t^*)-t^* v^* = -\log(2)$.
\end{proof}

\begin{lemma}\label{l:q-positif}
One has that $q_{\infty}(0)>0$.
\end{lemma} 
\begin{proof}
In Section \ref{s:building}, for $v^*-v$ small enough, we exhibit $c_- \in \H$ such that 
$T(c_-) \geq c_-$ and $c_-(x)>0$ for $x > 0$. We deduce (with Lemma \ref{l:depasse-pas}) that 
$T (c_-) (0)>0$. Now, letting $h:=T(c_-)$, Proposition \ref{p:monotonie} shows that $T(h) \geq h$. 
 Proposition \ref{p:convergence-sub} then yields that $q_{\infty}(0) \geq h(0)>0$. This conclusion is valid for small enough  $v^*-v$. 
 Since clearly\footnote{For instance by coupling. 
 } $q_{\infty}(0)$ is non-increasing with respect to $v$, the conclusion of the lemma is in fact valid for all $v<v^*$. 
\end{proof}

\begin{lemma}\label{l:encadrement}
There exists a constant $\kappa>0$ depending only on the distribution of $\zeta$, such that, for all $v<v^*$, 
$q_{\infty}(0) \geq \kappa q_{\infty}(1)$. 
\end{lemma}

\begin{proof}
Using the fact that $\psi(s) \geq s$ for all $x \in [0,1]$, we have that, for all $x \in [0,+\infty[$, 
$q(x) \geq E(q_{\infty}(x+\zeta-v))$. From Lemma \ref{l:depasse-pas}, we can find $\eta>0$ such that $P(\zeta \geq v^*+\eta)>0$.
Using the fact that $q$ is non-decreasing, we obtain that $q_{\infty}(x) \geq P(\zeta \geq v^*+\eta) q_{\infty}(x+\eta)$. Iterating, we see that, for all $n \geq 0$,  
$q_{\infty}0) \geq P(\zeta \geq v^*+\eta) ^n q_{\infty}(n \eta)$. Choosing $n$ large enough so that $n \eta \geq 1$, we get the desired result.
\end{proof}

We now prove Proposition \ref{p:charact-q}.
\begin{proof}[Proof of Proposition \ref{p:charact-q}]
Let $r \in \H$ satisfy $T(r)=r$ and $r \not\equiv 0$.  
According to Proposition \ref{p:convergence-sub}, we already have that 
 $r \leq q_{\infty}$. Our next step is to show that $r(0)>0$. To this end, let $$D := \{  x \in [0,+\infty[ ;     \       r(x)>0   \}.$$
 Since we assume that $r \not\equiv 0$, $D$ is non-empty, and since, moreover, 
$r$ is non-decreasing, $D$ has to be an interval, unbounded to the right. 
Since $v<v^*$, we know from Lemma \ref{l:depasse-pas} that $P(\zeta \leq v ) < 1$, whence the existence of $\eta>0$ such that 
$P(\zeta-v>\eta)>0$. Let $x$ be such that $x+\eta$ belongs to $D$.
Then, $x+\zeta-v$ belongs to $D$ with positive probability, so that $E(r(x+\zeta - v))>0$.
If, moreover, $x \geq 0$, we have that: 
$$
r(x)=T(r)(x)=\psi(E(r(x+\zeta-v)))
,$$
so that $r(x)>0$ since $\psi(s)>0$ for all $0<s \leq 1$.
We have therefore proved that $$(D-\eta)\cap [0,+\infty[ \subset D.$$
Since $D$ is a subinterval of $[0,+\infty[$, unbounded to the right, this implies that $D=[0,+\infty[$,  whence $r(0)>0$.
Now let $$F:=\{ \lambda \geq 0;    \forall x \in [0,+\infty[,     \      r(x)   \geq \lambda q_{\infty}(x) \}.$$
Since $q_{\infty}(0)>0$ by Lemma \ref{l:q-positif} and $r \leq q_{\infty}$,  $F$ must have the form $[0,\lambda_0]$ for some $\lambda_0 \in [0,1]$, and, we need to prove that indeed $\lambda_0=1$ to finish our argument. We first show that $\lambda_0>0$. Since $r$ is non-decreasing, we have that, for all $x \in \R$,  $r \geq r(0)\un_{[0,+\infty[}$, 
whence $r \geq r(0)q_{\infty}$ since $\un_{[0,+\infty[} \geq q_{\infty}$. As a consequence, $\lambda_0 \geq r(0)$, and we have seen that $r(0)>0$. 
Now, using the fact that $T(r)=r$, Proposition \ref{p:monotonie}, and the definition of $\lambda_0$, 
we see that
$$
r = T(r) \geq T(\lambda_0 q_{\infty}).
$$
For $x \geq 0$, we also have that $q_{\infty}(x)\ge q_{\infty}(0)>0$, and   $E(q_{\infty}(x+\zeta-v))>0$. Since $\lambda_0>0$, we can write:
$$
r(x)  \geq  \lambda_0 q_{\infty}(x) \; \textstyle{\frac{T(\lambda_0 q)(x)}{\lambda_0 q_{\infty}(x)}} =  \lambda_0 q_{\infty}(x) \;   \textstyle{\frac{T(\lambda_0 q_{\infty})(x)}{\lambda_0 T(q_{\infty})(x)}},$$
whence the inequality
$$r(x)  \geq \lambda_0 q_{\infty}(x) \; \chi \big(E(q_{\infty}(x+\zeta-v))\big),
$$
where $\chi$ is the map defined, for $s \in ]0,1]$, by
$$
\chi(s):=\frac{\psi(\lambda_0 s)}{\lambda_0 \psi(s)}=\frac{2-\lambda_0 s}{2-s},
$$
with the extension $\chi(0):=1$.
Since $q_{\infty}$ is non-decreasing, this is also the case of the map $x \mapsto E(q_{\infty}(x+\zeta-v))$.
Moreover,  $\chi$ too is non-decreasing on $[0,1]$, so we get that:
$$
r(x) \geq \lambda_0 \chi(E(q_{\infty}(\zeta-v))) q_{\infty}(x).
$$
Therefore, $\lambda_0 \chi(E(q_{\infty}(\zeta-v)))$ is an element of the set $F$. 
If $\lambda_0<1$, the fact that $E(q_{\infty}(\zeta-v))>0$ and strict monotonicity of $\chi$ show that $\chi(E(q(\zeta-v)))> \chi(0)=1$. We would thus have the existence of an element in $F$ strictly greater than $\lambda_0$, a contradiction. We thus conclude that $\lambda_0$ equals $1$.
Therefore one must have $r \geq q_{\infty}$. 
\end{proof}

\begin{proof}[Proof of Proposition \ref{p:convergence-super}]
Since $T(h) \leq h$, we deduce from Proposition \ref{p:monotonie} that, 
for all $x \in \R$, the sequence $(T^n(h)(x))_{n \geq 0}$ is non-increasing.
We deduce the existence of a map $T^{\infty}(h)$ in $\H$ such that, for all $x \in \R$,  
$T^{\infty}(h)(x) = \lim_{n \to +\infty}T^n(h)(x)$, and it is easily checked by dominated convergence that  $T(T^{\infty}(h))(x) = T^{\infty}(h)(x)$ for all $x \in \R$, while
$T^{\infty}(h)(x) \leq h(x)$. It remains to check that $T^{\infty}(h) \not\equiv 0$ to obtain the result, since Proposition \ref{p:charact-q} will then prove that $T^{\infty}(h) \equiv q_{\infty}$.
 Using the fact that $h$ is non-decreasing, we have that, for all $x \geq 0$, 
$h(x) \geq h(0) \un_{[0,+\infty[}(x)$. Moreover, it is easily checked that, for all $\lambda,s \in [0,1]$, 
$\psi(\lambda s) \geq \lambda \psi(s)$. Using Proposition \ref{p:monotonie}, we thus obtain that
$T^n(h)(x) \geq h(0) T^n(\un_{[0,+\infty[})(x)$, whence $T^{\infty}(h)(x) \geq h(0) q_{\infty}(x)$ by letting $n \to +\infty$. We deduce that $T^{\infty}(h)(x) \not\equiv 0$ since we have assumed that 
$h(0)>0$.
\end{proof}

\begin{remark}
The proof of Proposition \ref{p:charact-q} given above relies solely on analytical arguments. It is in fact possible to prove a slightly different version of Proposition  \ref{p:charact-q} -- which is sufficient to establish Propositions \ref{p:convergence-super} and \ref{p:convergence-sub} -- using a probabilistic argument based on the interpretation of the operator $T$ in terms of branching random walks. We thought it preferable to give a purely analytical proof here, since our overall proof strategy  for Theorem \ref{t:le-theoreme-mieux} is at its core an analytical approach.
\end{remark}

\section{Solving a linearized  equation}\label{s:lineaire}

According to Proposition \ref{p:charact-q}, the survival probability $q_{\infty}$ can be characterized as the unique  non-trivial solution (in the space $\H$) of the 
non-linear convolution equation, valid for $x \geq 0$:
\begin{equation}\label{e:encore-convol} r(x) = \psi ( E(r(x+\zeta-v))).\end{equation}
Linearizing the above equation around the trivial solution $r \equiv 0$, using the fact that $\psi(s) = 2s + o(s)$ as $s \to 0$, yields the following linear convolution equation: 
\begin{equation}\label{e:et-encore-convol} r(x) = 2 E(r(x+\zeta-v)).\end{equation}
As explained in Section \ref{s:building} below, explicit solutions to a slightly generalized version of (\ref{e:et-encore-convol}) are precisely what we use to build super- and sub-solutions to the original non-linear convolution equation (\ref{e:encore-convol}). To be specific, the linear convolution equation we consider are of the form
\begin{equation}\label{e:lin-prelim}  c(x) =  e^{-a} E(c(x+\zeta-v)),    \     x \in \R, \end{equation}
where $c   \    :     \     \R \to \C$ is a measurable map, $v$ is close to $v^*$ 
and $e^{-a}$ is close to $2$. 
Looking for solutions of (\ref{e:lin-prelim}) of the form  
\begin{equation}\label{e:type-de-sol} c(x) = e^{\phi x},\end{equation} 
where $\phi \in \C$, we see that a necessary and sufficient condition for (\ref{e:type-de-sol}) to yield a solution is that 
\begin{equation}\label{e:on-veut}E\left(e^{\phi (\zeta - v)  }\right)  = e^{a}.\end{equation}
Due to our initial assumption on $\zeta$, we have that 
$$E\left(e^{t^* (\zeta - v^*)  }\right)  = 1/2,$$
so that one can hope to find solutions to (\ref{e:on-veut}) by performing a perturbative analysis. This is precisely what is done in Section \ref{ss:perturb}.
Then, in Section \ref{ss:analyse}, some of the properties of the corresponding solutions of (\ref{e:lin-prelim}) are studied.

\subsection{Existence of exponential solutions}\label{ss:perturb}

Consider the extension of the Laplace transform of $\zeta$ when the parameter $t \in \C$.
Since $\zeta$ has bounded support, $t \mapsto E(\exp(t \zeta))$ defines a holomorphic map from $\C$ to $\C$, and, since $E(\exp(t^* \zeta)) \notin \R_-$, we may extend the definition of 
$\Lambda$ to an open neighborhood $U$ of $t^*$ in $\C$, by setting $\Lambda(t) := \log  E(\exp(t \zeta))$ for $t \in U$, using the principal determination of the logarithm. We thus  obtain a holomorphic map on $U$.

\begin{prop}\label{p:implicite}
Let a be a holomorphic function defined on a neighborhood of zero, such that $a(0)=-\log(2)$ and $a'(0)=0$.
There exists $\epsilon_0>0$  and a map $\phi  \     :      \       [ 0, \epsilon_0 ] \to U$ such that, for all
$v \in [v^*-\epsilon_0,v^*]$, the following identity holds:
\begin{equation}\label{e:voulue-1}E\left(e^{\phi(v^*-v) (\zeta - v)  }\right)  = e^{a(v^*-v)},\end{equation}
and such that, as $\epsilon \to 0$,   
\begin{equation}\label{e:voulue-2}\phi(\epsilon) = t^* +  i   \sqrt{  \frac{2 t^*\epsilon}{\Lambda''(t^*)}    } + O(\epsilon).\end{equation}
\end{prop}

 \begin{proof}
 Given $v$ in the vicinity of $v^*$, we are looking for a $t \in U$ such that 
 \begin{equation}\label{e:resolution-1}E\left(e^{t (\zeta - v)  }\right)  = e^{a(v^*-v)}.\end{equation}
For $t$ close enough to $t^*$, and $v$ close enough to $v^*$, we can use the logarithm and observe that the equation
 \begin{equation}\label{e:resolution-2} \Lambda(t) - t v = a(v^*-v).\end{equation}
is sufficient for (\ref{e:resolution-1}) to hold.
Expanding $\Lambda$ for $t \in U$, we have that 
 $$\Lambda(t) = \Lambda(t^*) + (t-t^*) \Lambda'(t^*) + (t-t^*)^2 g(t),$$ 
where $g$ is holomorphic on $U$ and satisfies $g(t^*)=\Lambda''(t^*)/2$.
Plugging (\ref{e:lambda-1}) and (\ref{e:lambda-2}), we can rewrite the above expansion as
 $$\Lambda(t) =  -\log(2)+t v^*+ (t-t^*)^2 g(t).$$
On the other hand, using our assumptions on $a$, we can write
$$a(v^*-v)=-\log(2) + (v^*-v) b(v^*-v),$$
where $b$ is an holomorphic function in a neighborhood of zero such that $b(0)=0$.
Finally,  (\ref{e:resolution-2}) reads 
 \begin{equation}\label{e:resolution-3}  (t-t^*)^2 g(t) =  (v-v^*) (t- b(v^*-v)).\end{equation}
Observe that, since $g(t^*) = \Lambda''(t^*)/2 \notin \R_-$, we can define 
$\sqrt{g(t)}$ for $t$ close to $t^*$, using the principal determination of the logarithm and the 
definition $\sqrt{z} = \exp(\log z/2)$.
We can similarly define $\sqrt{t - b(z)}$ for $(t,z)$ close to $(t^*,0)$.
Now,  for $(t,u) \in \C \times \C$ in the vicinity of $(t^*,0)$, consider the equation . 
   \begin{equation}\label{e:resolution-4}  (t-t^*) \sqrt{g(t)}  = u  \sqrt{t - b(-u^2)}.\end{equation}
 Clearly, if (\ref{e:resolution-4}) holds with $u=i \sqrt{v^*-v}$ when $v^*$ and $v$ are
 real numbers such that $v<v^*$, then  (\ref{e:resolution-3}) holds.
 Now consider the map $\Xi$ defined in the neighborhood of $(t^*,0)$ in $\C \times \C$ by 
$$\Xi(t,u) :=  (t-t^*) \sqrt{g(t)}  - u  \sqrt{t - b(-u^2)}.$$
Observe that $\Xi(t^*,0) = 0$, and that the (holomorphic) derivative of $\Xi$ with respect to 
$t$ at $(t^*,0)$ is equal to $\sqrt{g(t^*)} = \sqrt{  \Lambda''(t^*)/2 } \neq 0$.
 Identifying $\C$ with $\R \times \R$, we can thus view $\Xi$ as a smooth map defined on an open set
 of $\R^4$, and apply the implicit function theorem to deduce the existence of a smooth 
 map $f$ defined on a neighborhood of $0$ in $\C$ such that $f(0)=t^*$ and such that, 
 for all $u$ near zero, 
\begin{equation} \label{e:implicite}  \Xi(f(u),u)=0.\end{equation}  We now set $\phi(\epsilon) = f(i \sqrt{\epsilon})$, which yields (\ref{e:voulue-1}).
 Then, one obtains $(\ref{e:voulue-2})$ by computing the derivative of $f$ at zero from 
 (\ref{e:implicite}) in the usual way.
  \end{proof}

 \subsection{Properties of the exponential solutions}\label{ss:analyse}
 
 Now let $a$, $\epsilon_0$ and $\phi$ be given as in Proposition~\ref{p:implicite}.
 Throughout the sequel, we use the notation 
  $$\epsilon = v^*-v.$$
  An immediate consequence of the proposition is that, for all $v \in [v^*-\epsilon_0,v^*]$, 
 the map defined on $\R$ by $x \mapsto e^{\phi(\epsilon)x}$, solves the equation 
 \begin{equation}\label{e:lin}  c(x) =  e^{-a(\epsilon)} E(c(x+\zeta-v)).\end{equation}
If $a(\epsilon) \in \R$ when $\epsilon \in \R$ (this will be the case in all the examples we consider below), then the map   $x \mapsto e^{\overline{\phi(\epsilon)}x}$ is also a solution of 
(\ref{e:lin}), where $\overline{z}$ denotes the conjugate complex number of $z$.
Let us set $\alpha(\epsilon):=\Re(\phi(\epsilon))$ and $\beta(\epsilon):=\Im(\phi(\epsilon))$. 
Thus, we obtain a solution of (\ref{e:lin}) if we set 
\begin{equation}\label{e:la-solution}d(x) := e^{\alpha(\epsilon)x} \sin ( \beta(\epsilon) x).\end{equation}

Consider $\epsilon$ small enough so that $\alpha(\epsilon)>0$ and $\beta(\epsilon) > 0$. 
Note that $$d(0)=  d(\pi/\beta(\epsilon))=  0,$$ 
and that one has
$$ \left\{ \begin{array}{l}
d \leq 0 \mbox{ on } [   -\pi/\beta(\epsilon),    0],\\
d \geq 0 \mbox{ on } [ 0  , \pi/\beta(\epsilon)],\\  
d \leq 0 \mbox{ on } [ \pi/\beta(\epsilon), 2\pi/\beta(\epsilon)].
\end{array}\right.$$

The derivative of $d$ is given by $$d'(x) =  \alpha(\epsilon) e^{\alpha (\epsilon) x} \sin(\beta (\epsilon) x) +  \beta(\epsilon) e^{\alpha (\epsilon) x } \cos(\beta (\epsilon)x).$$
One thus checks that $d$ attains a unique maximum on the interval 
$[0, \pi/\beta(\epsilon)]$, at a value $x=L(\epsilon)$ satisfying
\begin{equation} \label{e:prendre-la-tangente}\tan(\beta(\epsilon) L(\epsilon)) = -\beta(\epsilon)/\alpha(\epsilon),\end{equation}
and that $d$ is increasing on the interval $[0,L(\epsilon)]$.
As $\epsilon$ goes to $0$, we know from (\ref{e:voulue-2})
that 
\begin{equation}\label{e:alpha-et-beta}\alpha(\epsilon) = t^* + O(\epsilon) ,      \      \beta(\epsilon) =     \sqrt{  \frac{2 t^*\epsilon}{\Lambda''(t^*)}    }  +   O(\epsilon),\end{equation}
and we then deduce from (\ref{e:prendre-la-tangente}) that, as $\epsilon$ goes to $0$, 
\begin{equation}\label{e:limite-L-1}L(\epsilon)    =   \frac{\pi}{\beta(\epsilon)} - 1/t^* +o(1),\end{equation}
whence 
\begin{equation}\label{e:limite-L-2}L(\epsilon)   =  \pi  \sqrt{  \frac{\Lambda''(t^*)} {2 t^*\epsilon}   } + O(1).\end{equation}

\section{Building a super-solution and a sub-solution to the original equation}\label{s:building}

In this section, we explain how to transform the explicit solutions of the linear equation (\ref{e:lin-prelim}) 
obtained in the previous section, into super- and sub- solutions of the non-linear convolution equation of Section \ref{s:equations}. In the sequel, $\Delta$ denotes a real number (whose existence is guaranteed by the assumption that $\zeta$ has bounded support) such that, for all $v$ in the vicinity of $v^*$, 
 $$P(|\zeta - v| \leq \Delta) = 1.$$ 

\subsection{The super-solution}


Let us choose the function $a$ in Proposition~\ref{p:implicite} as the constant function 
$a(\epsilon):=-\log(2)$. Then consider the function $d(\cdot)$ defined in (\ref{e:la-solution}), and note that any function of the form 
$$c(x) := A(\epsilon) d(x + \Delta), \  x \in \R$$
  is a solution of (\ref{e:lin-prelim}). 
Now let $A(\epsilon)$ be implicitly defined by the requirement that 
$$c(L(\epsilon)-2\Delta)=1.$$
This last condition rewrites more explicitly as
\begin{equation}\label{e:def-A-1}A(\epsilon) e^{\alpha(\epsilon) (L(\epsilon)-\Delta)} \sin ( \beta(\epsilon) ( L(\epsilon) - \Delta) ) = 1. \end{equation}
From (\ref{e:limite-L-1}), we have that, as $\epsilon \to 0$,  
$$\sin ( \beta(\epsilon) (L(\epsilon)-\Delta)) \sim \beta(\epsilon) (1/t^*+\Delta).$$
Combining with (\ref{e:def-A-1}), (\ref{e:limite-L-2}) and (\ref{e:alpha-et-beta}), we deduce that, as $\epsilon \to 0$, 
$A(\epsilon)>0$ and
\begin{equation}\label{e:lim-A-1}\log A(\epsilon) = -  \pi  \sqrt{  \frac{\Lambda''(t^*) t^*} {2 \epsilon}} + O(\log \epsilon).\end{equation}
For notational convenience, we introduce $$C(\epsilon):=L(\epsilon)-2 \Delta.$$
The next proposition summarizes  the properties of $c$ that we shall use in the sequel.
\begin{prop}\label{p:des-proprietes-super}
 For small enough $\epsilon$, the following properties hold.
\begin{itemize}
\item[(i)] $c(x) \geq 0$ for all $x \in [-\Delta, 0]$; 
\item[(ii)] $0 \leq c(x) \leq 1$ for all $x \in [0,C(\epsilon)]$;
\item[(iii)] $c(x) \geq 1$ for all $x \in [C(\epsilon),C(\epsilon)+\Delta]$.
\item[(iv)] $c$ is non-decreasing on $[0,C(\epsilon)]$
\end{itemize}
Moreover, as $\epsilon$ goes to zero, 
 \begin{equation}\label{e:borne-sup}\log c(1) =  -  \pi  \sqrt{  \frac{\Lambda''(t^*) t^*} {2 \epsilon}}+O(\log \epsilon).\end{equation}
\end{prop}

\begin{proof}
(i), (ii), (iii) and (iv) are rather direct consequences of the definition and of the analysis of Section \ref{ss:analyse}. As for (\ref{e:borne-sup}), it is readily derived from (\ref{e:lim-A-1}) and (\ref{e:alpha-et-beta}).
\end{proof}

We now define the map $c_+   \      :       \     \R \to  [0,1]$ by 
\begin{itemize}
\item $c_+(x) = 0$ for all $x < 0$; 
\item $c_+(x)=c(x)$ for all $x \in [0,L(\epsilon)-2 \Delta]$;
\item $c_+(x) = 1$ for all $x > L(\epsilon)-2 \Delta$.
\end{itemize}

Now let $\psi_+$ be defined on $[0,1]$ by $$\psi_+(s):=\min(2s, 1).$$
\begin{prop}\label{p:compare-sup}
For all $s \in [0,1]$, $\psi_+(s) \geq \psi(s)$.
\end{prop}
\begin{proof}
Immediate.
\end{proof}

\begin{prop}\label{p:super-solution}
For small enough $\epsilon$, one has that $c_+ \in \H$ and, for all $x \geq 0$, 
$$c_+(x)    \geq   \psi_+(E(c_+(x+\zeta-v))).$$
 \end{prop}
\begin{proof}
The fact that $c_+ \in \H$ is guaranteed by the definition and properties (i) to (iv) of Proposition \ref{p:des-proprietes-super}.
Consider first the case $x \in [0,C(\epsilon)]$. By definition, one has that
$$c(x) = 2E(c(x+\zeta-v)).$$
Since, by construction, $c_+ \leq c$ on $[-\Delta, C(\epsilon)+\Delta]$, we deduce that 
$$c(x) \geq 2E(c_+(x+\zeta-v)).$$
Since $\psi_+(s) \leq 2s$ for all $s \in [0,1]$, we deduce that
$$c(x) \geq \psi_+(E(c_+(x+\zeta-v))),$$
whence, 
remembering that $c_+(x)=c(x)$, 
$$c_+(x) \geq \psi_+(E(c_+(x+\zeta-v))).$$
Now, for $x > C(\epsilon)$, we have that $c_+(x)=1$, so that $c_+(x) \geq \psi_+(s)$ for all $s \in [0,1]$.
In particular,  $$c_+(x) \geq \psi_+(E(c_+(x+\zeta-v))).$$
\end{proof}

\begin{coroll}\label{c:une-super-solution}
For small enough $\epsilon$, one has that $c_+ \geq T(c_+)$.
\end{coroll}
\begin{proof}
Direct consequence of Propositions \ref{p:compare-sup} and \ref{p:super-solution}. 
\end{proof}

\subsection{The sub-solution}

In this section, we  choose the function $a$ in Proposition~\ref{p:implicite} as the function 
$a(\epsilon):=-\log(2)+\epsilon^2$. For notational convenience, we also introduce
$$\gamma(\epsilon):=e^{-a(\epsilon)},       \           h(\epsilon) := \gamma(\epsilon)(2-\gamma(\epsilon)).$$
As above, we use the function $d(\cdot)$ defined in (\ref{e:la-solution}), and let 
$$c(x) := A(\epsilon) d(x) = A(\epsilon) e^{\alpha(\epsilon)x} \sin ( \beta(\epsilon) x),$$
where $A(\epsilon)$ is implicitly defined by the requirement that $$c(L(\epsilon))=h(\epsilon).$$
The same argument as in the previous section then shows that, as $\epsilon \to 0$, $A(\epsilon)>0$ and
\begin{equation}\label{e:lim-A-2}\log A(\epsilon) = -  \pi  \sqrt{  \frac{\Lambda''(t^*) t^*} {2 \epsilon}} + O(\log \epsilon).\end{equation}
Here are the properties of $c$ that we shall use in the sequel.
\begin{prop}\label{p:des-proprietes-sub}
 For small enough $\epsilon$, the following properties hold.
\begin{itemize}
\item[(i)] $c(x) \leq 0$ for all $x \in [-\Delta, 0]$; 
\item[(ii)] $0 \leq c(x) \leq h(\epsilon)$ for all $x \in [0,L(\epsilon)]$;
\item[(iii)] $c(x) \leq h(\epsilon)$ for all $x \in [L(\epsilon),L(\epsilon)+\Delta]$;
\item[(iv)] $c$ is non-decreasing on $[0,L(\epsilon)]$.
\end{itemize}
Moreover, as $\epsilon$ goes to zero, 
 \begin{equation}\label{e:borne-inf}\log c(1) =  -  \pi  \sqrt{  \frac{\Lambda''(t^*) t^*} {2 \epsilon}} + O(\log \epsilon).\end{equation}
\end{prop}

\begin{proof}
(i), (ii), (iii) and (iv) are rather direct consequences of the definition and of the analysis of Section \ref{ss:analyse}. As for (\ref{e:borne-inf}), it is readily derived from (\ref{e:lim-A-2}) and (\ref{e:alpha-et-beta}).
\end{proof}

We define the map $c_-   \      :       \     \R \to  [0,1]$ by 
\begin{itemize}
\item $c_-(x) = 0$ for all $x < 0$; 
\item $c_-(x)=c(x)$ for all $x \in [0,L(\epsilon)]$;
\item $c_-(x) = h(\epsilon)$ for all $x > L(\epsilon)$.
\end{itemize}

Now define $\psi_-$  on $[0,1]$ by 
$$\psi_-(s):=\min(\gamma(\epsilon)s, h(\epsilon)).$$
\begin{prop}\label{p:compare-sub}
For small enough $\epsilon$, for all $s \in [0,1]$, $\psi_-(s) \leq \psi(s)$.
\end{prop}
\begin{proof}
Immediate.
\end{proof}

\begin{prop}\label{p:sub-solution}
For small enough $\epsilon$, one has $c_- \in \H$ and, for all $x \geq 0$,  
$$c_-(x)    \leq   \psi_-(E(c_-(x+\zeta-v))).$$
 \end{prop}
\begin{proof}
The fact that $c_- \in \H$ is a direct consequence of the definition and of properties (i) to (iv) of Proposition \ref{p:des-proprietes-sub}.
Consider first the case $x \in [0,L(\epsilon)]$. By construction, one has that
$$c(x) = \gamma(\epsilon)E(c(x+\zeta-v)).$$
Since, by construction, $c_- \geq c$ on $[-\Delta, L(\epsilon)+\Delta]$, we deduce that 
\begin{equation}\label{e:comparaison-x-2}c(x) \leq \gamma(\epsilon)E(c_-(x+\zeta-v)).\end{equation}
Since we have $c(x) \leq h(\epsilon)$ and $c_-(x)=c(x)$, we deduce that  
$$c_-(x) \leq  \psi_-(E(c_-(x+\zeta-v))).$$
Consider now  $x > L(\epsilon)$. Since $c_-$ is non-decreasing,  
so is $y \mapsto \gamma(\epsilon)E(c_-(y+\zeta-v))$.  We deduce that
$$\gamma(\epsilon)E(c_-(x+\zeta-v)) \geq \gamma(\epsilon)E(c_-(L(\epsilon)+\zeta-v)).$$ 
Using again the fact that $c_- \geq c$ on $[-\Delta, L(\epsilon)+\Delta]$, we have that 
$$\gamma(\epsilon)E(c_-(L(\epsilon)+\zeta-v)) \geq \gamma(\epsilon)E(c(L(\epsilon)+\zeta-v)) = c(L(\epsilon)) = h(\epsilon).$$
Since $c_- \leq h(\epsilon)$, we finally deduce that  
$$c_-(x) \leq \psi_-( E(c_-(x+\zeta-v))).$$
\end{proof}

\begin{coroll}\label{c:une-sub-solution}
For small enough $\epsilon$, one has that $c_- \leq T(c_-)$.
\end{coroll}
\begin{proof}
Direct consequence of Propositions \ref{p:compare-sub} and \ref{p:sub-solution}. 
\end{proof}

\section{Proof of Theorem \ref{t:le-theoreme-mieux}}\label{s:conclusion}

We now put together the different pieces leading to the proof of Theorem \ref{t:le-theoreme-mieux}. For small enough $\epsilon$, Corollary \ref{c:une-super-solution} yields a super-solution of the non-linear convolution equation, i.e. a function $c_+ \in \H$ such that 
$c_+ \geq T(c_+)$. By Proposition \ref{p:convergence-super},  we automatically have that 
$q_{\infty} \leq c_+$. Now we have $q_{\infty}(0) \leq q_{\infty}(1) \leq c_+(1)$, so that the asymptotic behavior stated in (\ref{e:borne-sup}) yields that 
$$\log q_{\infty}(0) \leq -  \pi  \sqrt{  \frac{\Lambda''(t^*) t^*} {2 (v^*-v) }}+O(\log (v^*-v)).$$
Similarly, Corollary \ref{c:une-sub-solution} yields a sub-solution, i.e. a function $c_- \in \H$ such that 
$c_- \leq T(c_-)$, and, by Proposition \ref{p:convergence-sub} we have that 
$q_{\infty} \geq c_-$. Using Lemma \ref{l:encadrement}, we obtain that $q_{\infty}(0) \geq \kappa q_{\infty}(1) \geq \kappa c_-(1)$, so that the asymptotic behavior stated in (\ref{e:borne-inf}) yields that 
$$\log q_{\infty}(0) \geq -  \pi  \sqrt{  \frac{\Lambda''(t^*) t^*} {2(v^*-v)}}+O(\log (v^*-v)).$$
This concludes the proof of Theorem \ref{t:le-theoreme-mieux}.

\begin{remark}
One may wonder whether it is possible to extend our method of proof to a more general situation such as the one treated in \cite{GanHuShi}, i.e. a fairly general supercritical stochastic branching mechanism, and random walk steps whose Laplace transform is finite in a neighborhood of zero. We expect our proof to be robust to more general branching mechanisms, since the key properties of the resulting non-linear convolution equation (monotonicity, qualitative features of the linearized equation) should be preserved. On the other hand, it might not be easy to replace the bounded support assumption by a less stringent one, since one would have to control how large unlikely values of the random walk steps affect the validity of the construction of super- and sub-solutions described in Section \ref{s:building}.
\end{remark}

\begin{remark}
The error term in the statement of Theorem \ref{t:le-theoreme-mieux} is $O(log (v^*-v))$,  but we suspect that its correct order of magnitude might be $O(1)$, as suggested  by the (mathematically non-rigorous) results obtained by B. Derrida and D. Simon in  \cite{DerSim1, DerSim2}, and by analogous results obtained in the context of branching Brownian motion by J. Berestycki, N. Berestycki and J. Schweinsberg \cite{BerBerSch2} and E. Aidekon, S. Harris, and R. Pemantle \cite{AidHarPem}. 
Indeed, pushing the calculations just a little further with our method, it is not difficult to show that the error term is bounded above by a $O(1)$, i.e. that 
$$\log \P_0(A_{\infty}(v)) \leq -  \pi  \sqrt{  \frac{\Lambda''(t^*) t^*} {2 (v^*-v)}} + O(1).$$
However, an extra $\log(v^*-v)$ creeps into our lower bound, as a consequence of the sub-solution in Section \ref{s:building} being built with $a(\epsilon):=-\log(2)+\epsilon^2$ instead of $a(\epsilon):=2$, so that the best lower bound we can prove is $O(\log (v^*-v))$.
\end{remark}

\section{Connection with the Brunet-Derrida theory of stochastic fronts}\label{s:BD}

  Through a series of papers (see e.g. \cite{BruDer1, BruDer2, BruDer3, BruDerMueMun, BruDerMueMun2, BruDerMueMun3}), Brunet and Derrida, partly with Mueller and Munier,  developed a theory for 
the propagation of stochastic fronts described by perturbed F-KPP like equations. The methods and arguments used by Brunet and Derrida are not mathematically rigorous, but several of the corresponding results have now been given regular mathematical proofs (see e.g.  \cite{BenDep2, BenDep1, DumPopKap,ConDoe, MueMytQua, MueMytQua2, BerGou, DurRem, BerBerSch}). The purpose of this section is to explain how the present paper relates to these developments. 

We first explain how the survival probability of the branching random walk is related to travelling wave solutions of  F-KPP like equations. We then explain how such travelling waves arise in the context of the Brunet-Derrida theory of stochastic fronts, and describe the approach used by Brunet and Derrida to deal with these travelling waves. 
We finally discuss how Theorem \ref{t:le-theoreme-mieux}, its proof, and some of its consequencs, fit into the series of rigorous mathematical results establishing predictions from the Brunet-Derrida theory.

\subsection{Survival probability and travelling waves}

The F-KPP equation, named after Fisher \cite{Fis} and Kolmogorov, Petrovsky and Piscounov \cite{KolPetPis}, 
is one of the classical PDE models of front propagation, whose salient feature is to lead to travelling wave solutions.  
In its simplest form, the equation reads 
$$\frac{\partial u}{\partial t}  =  \frac{\partial^2 u}{\partial x^2}  + u(1-u),$$
where $u=u(x,t)$, $x \in \R$,  $t \geq 0$, and a travelling wave solution means a solution $u$ of the form $$u(x,t) = g(x-vt),$$ where $v \in \R$ is the wave speed, and $g(x),  \ x \in \R$  describes the wave shape.

In our context, it turns out that the survival probability $q_{\infty}$ (as a function of the starting point $x$ of the branching random walk) can be viewed as the shape of a travelling wave solution to a discrete analog of the F-KPP equation, with a special boundary condition.  Indeed, define, for all $n \geq 0$ and $x \in \R$,  
$$u_{n}(x) := q_{\infty}(-x+nv).$$ 
From Proposition \ref{p:iter}, one sees that the following equation holds for $u_n$:
for all $x \leq (n+1) v$, 
\begin{equation}\label{e:fkpp} u_{n+1}(x) = \psi(E(u_{n}(x-\zeta))).\end{equation}
Using the definition of $\psi$, and rearranging the formulas a little, the above equation rewrites\footnote{Written below the equation is the term-by-term analogy with the F-KPP equation.}
\begin{equation}\label{e:F-KPP-discret}\underbrace{u_{n+1}(x) -u_{n}(x)}_{\leftrightsquigarrow \frac{\partial u}{\partial t} } =  \underbrace{E (u_{n}(x-\zeta) ) - u_n(x)}_{\leftrightsquigarrow  \frac{\partial^2 u}{\partial x^2} }    +   \underbrace{ E(u_{n}(x-\zeta)) -  E(u_{n}(x-\zeta))^2}_{\leftrightsquigarrow u - u^2}.\end{equation} 
Remember that the above equation holds only for $x \leq (n+1)v$, while, due to the fact that $q_{\infty} \equiv 0$ on $]-\infty,0[$, we have
 that $u_{n+1}(x) = 0 $  for $x > (n+1)v$.
  It is now apparent that  $(n,x) \mapsto u_n(x)$ describes a discrete travelling wave with speed $v$, obeying a discrete analog of the F-KPP equation at the left of the front, and identically equal to zero at the right of the front, where at time $n$ the location of the front is $nv$.

\subsection{Brunet-Derrida theory of stochastic fronts}  
  
We begin by describing two distinct models considered by Brunet and Derrida. One is the stochastic F-KPP equation  
\begin{equation}\label{e:F-KPP-noise}\frac{\partial u}{\partial t}  =   \frac{\partial^2 u}{\partial x^2} + u(1-u)+\sqrt{\frac{u(1-u)}{N}} \dot{W},\end{equation} 
where $\dot{W}$ is a standard space-time white-noise,  and $N$ is large, with the initial condition $u(x,t=0) = \un_{]-\infty,0]}(x)$. 

The other is a particle system on the real line where a population of $N$ particles evolves according to repeated steps of branching and selection. Branching steps are identical to those of the branching random walk considered in this paper, each particle in a given generation being replaced in the next generation by two new particles whose locations are shifted using two independent random walk steps. Selection steps consist in keeping only the $N$ rightmost particles among the $2N$ obtained by branching from the population in the previous generation. 

It is possible to see these two models as describing the propagation of a front, and Brunet and Derrida found (see \cite{BruDer1, BruDer2, BruDer3}) that, for both models, the limiting velocity $v_N$ of the front has the following behavior as $N$ goes to infinity: 
\begin{equation}\label{e:BD-shift}v_{\infty} - v_N \sim C (\log N)^{-2},\end{equation}
where $v_{\infty}$ is the limiting value of $v_N$ as $N$ goes to infinity. Mathematical proofs of these results were then obtained in \cite{MueMytQua} for the stochastic F-KPP equation  case, and in \cite{BerGou} for the branching-selection particle system case.  
Let us mention that the description of stochastic fronts obtained by Brunet, Derrida, Mueller and Munier  goes far beyond (\ref{e:BD-shift}), which is, in some sense, a first-order result  (we refer to \cite{BruDerMueMun, BruDerMueMun2, BruDerMueMun3} for more details).

A very rough sketch of the argument used by Brunet and Derrida  to deduce (\ref{e:BD-shift}) is as follows. For solutions $u$ of the stochastic F-KPP equation (\ref{e:F-KPP-noise}),  at every time $t>0$, $x \mapsto u(x, t)$ continuously connects $1$ at $x=-\infty$ to $0$ at some random $x=X(t)$ defining the position of the front, right of which $u(\cdot, t)$ is identically zero.  Looking at the equation (\ref{e:F-KPP-noise}), one can see that stochastic effects due to the noise term counterbalance the $u(1-u)$ creation term when $u$ is of order $1/N$. 
To find the asymptotic speed of propagation of the front, one should thus look for travelling waves  obeying the F-KPP equation at the left of the front, taking values of order $1/N$ near the front, and which are identically equal to zero at the right of the front. To study these travelling waves, one replaces the F-KPP equation by a linear approximation, for which explicit solutions can be found -- these solutions should be approximately valid for the original equation, thanks to the fact that the values of $u$ are small near the front. One can then check that the speed of these travelling waves must satisfy (\ref{e:BD-shift}). The same line of argument is used to derive (\ref{e:BD-shift}) for the branching-selection particle system. In this case, the $N \to +\infty$ limit of the model can be viewed as a time-discrete version of the F-KPP equation similar to (\ref{e:F-KPP-discret}), and the actual system with $N$ particles is described by a perturbation of this equation. As in the case  of the stochastic F-KPP equation, one finds that the relevant scale for perturbations is $1/N$, but this time this is due to the fact that a population of $N$ particles has a resolution of $1/N$ for representing a probability mass.

Now, the key observation is that the dicrete travelling wave $(n,x) \mapsto u_n(x)$ deduced from $q_{\infty}$ in the previous section corresponds just to the object 
studied by Brunet and Derrida in their argument for (\ref{e:BD-shift}), namely travelling waves obeying an F-KPP like equation  at the left of the front, that are identically equal to zero at the right of the front. One slight difference is that Brunet and Derrida prescribe the order of magnitude of the values of the travelling wave near the front to be $1/N$, and ask for the corresponding speed, while we prescribe the speed to be $v=v^*-\epsilon$, and ask for the corresponding order of magnitude of the values near the front.

 Looking for a speed $v(N)$ such that, for large $N$,  
  $$\P_0(A_{\infty}(v(N))) \sim 1/N,$$    
  we see from Theorem \ref{t:le-theoreme} that one must have 
  \begin{equation}v^* - v(N) \sim  \textstyle{\frac{\pi^{2}}{2}} t^{*} \Lambda''(t^{*}) (\log N)^{-2},\end{equation}
  which corresponds precisely to the behavior (\ref{e:BD-shift}) of $v_N$ obtained by Brunet and Derrida in the branching-selection particle system case, and is also the key to the rigorous proof of (\ref{e:BD-shift}) given in \cite{BerGou}.

  The connection between the survival probability of the branching random walk and the travelling wave solutions of perturbed discrete F-KPP equations of the type investigated by Brunet and Derrida, was in fact used by Derrida and Simon in \cite{DerSim1, DerSim2} to derive Theorems \ref{t:le-theoreme} and \ref{t:le-theoreme-mieux}, in a mathematically non-rigorous way, from the Brunet and Derrida approach outlined above.

\subsection{Discussion}

Our strategy for proving Theorem \ref{t:le-theoreme-mieux} is based on the original argument of Brunet and Derrida sketched in the previous section, and turns out to be quite different from the probabilistic approach used by Gantert, Hu and Shi in \cite{GanHuShi}. Remember that the idea is  to replace the discrete F-KPP equation by a linear approximation of it for which explicit solutions can be computed. In our context, this corresponds to replacing the non-linear convolution equation (\ref{e:encore-convol}) by the linear equation (\ref{e:lin-prelim}).  An additional idea we use is that the monotonicity properties of the non-linear equation allow for a rigorous comparison between suitably adjusted solutions of the linear equation, and solutions of the original non-linear one, via the construction of super- and sub-solutions. We were inspired by the work of Mueller, Mytnik and Quastel \cite{MueMytQua, MueMytQua2}, where this comparison idea is used as an intermediate step in the rigorous proof of (\ref{e:BD-shift}) in the stochastic F-KPP equation case. In their context, the equation corresponding to  (\ref{e:encore-convol}) is the non-linear second-order ordinary differential equation $ -v u'  = u'' + u(1-u)$, and one can rely on specific techniques, such as phase-plane analysis, to implement the comparison idea. However, these tools are not available in our discrete-time setting, and we had to find a quite different way of achieving the comparison argument. Note that, as a direct by-product of the argument in  \cite{MueMytQua, MueMytQua2}, one can prove an analog to 
Theorem \ref{t:le-theoreme-mieux} for the survival probability of the branching Brownian motion killed below a linearly moving boundary.

The asymptotic behavior of the survival probability $q_{\infty}$ plays a key role in our proof of  (\ref{e:BD-shift}) in the branching-selection particle system case, which was 
given in \cite{BerGou}, and relied on the proof of Theorem \ref{t:le-theoreme} given by Gantert, Hu and Shi in \cite{GanHuShi}.  
 In this regard, an interesting feature of the proof of Theorem \ref{t:le-theoreme-mieux} presented here is that, combined to the comparison argument described in \cite{BerGou},
 it provides a proof of  (\ref{e:BD-shift})  in the branching-selection particle system case which is along the lines of the original argument by Brunet and Derrida.       
 What is (a little) more, the slight improvement from Theorem \ref{t:le-theoreme} to Theorem \ref{t:le-theoreme-mieux} concerning the order of magnitude of the error term allows us to refine  (\ref{e:BD-shift}) and thus achieve a result comparable with the one obtained in \cite{MueMytQua, MueMytQua2}. Indeed, under the assumptions of Theorem \ref{t:le-theoreme-mieux}, one has that, in the branching-selection particle system case,  
$$v^* - v_{N} - \textstyle{\frac{\pi^{2}}{2}} t^{*} \Lambda''(t^{*}) (\log N)^{-2} =   O\left(  \frac{\log \log N}{(\log N)^3}    \right).$$
 Note that the $ \frac{\log \log N}{(\log N)^3}$ term in the above equation corresponds to the actual order of magnitude expected from \cite{BruDerMueMun}, where more precise predictions are given.

%
%
%

\bibliographystyle{plain}
\bibliography{BD-survie}

\appendix
\section{The $t^*$ assumption}

Remember that the core results in this paper are derived under the following two assumptions on the random variable $\zeta$ describing the steps of the branching random walk:
\begin{enumerate}
\item the probability distribution of $\zeta$ has bounded support;
\item there exists $t^*>0$ such that $\Lambda(t^*)-t^* \Lambda'(t^*) = -\log(2)$, where $\Lambda$ is the log-Laplace transform of $\zeta$.
\end{enumerate}

The  goal of the present section is to discuss the role and meaning of the second assumption, which we call the $t^*$ assumption. For the sake of simplicity, we keep assuming throughout this section that  $\zeta$ is bounded.
We first try to provide some intuition about the $t^*$ assumption by giving some alternative characterizations.
We then discuss the asymptotic behavior of the survival probability when the $t^*$ assumption is not met.

\subsection{Alternative characterizations of the $t^*$ assumption}

The results contained in this section seem to be more or less folklore in the branching process literature. However, we failed at finding a reference providing both complete statements and self-contained proofs, so we decided to include the following short elementary account of these results.

We start by giving a general characterization of the critical speed $v^*$. 
Let $\zeta_+$ be the essential supremum of $\zeta$, and let $\zeta_-$ be its essential infimum.
Then let $\theta:\R\to\R$ be defined by 
$$
\theta(t)=\log \big(2E(\exp(t\zeta))\big)=\log (2)+\Lambda(t).
$$
From the bounded support assumption for $\zeta$, $\theta$ is well-defined and finite for all $t \in \R$, and is
 $C^{\infty}$ and convex as a function of the parameter $t$.
Then define
\begin{equation}\label{e:def-de-v-astre}
v^*=\inf\left\{\frac{\theta(t)}{t}, t>0\right\}.
\end{equation}



From results by J. M. Hammersley \cite{Hammersley-subadditive}, J. F. C. Kingman \cite{Kingman-branching-75}, and J. D. Biggins \cite{Biggins-76}, $v^*$ appears as the maximum limiting speed in the branching random walk. Indeed, letting $M_n$ denote the position of the rightmost particle of the branching random walk at time $n$, one has that 
\begin{theorem} With probability one,
$\frac{M_n}{n} \to v^*$.
\end{theorem}
We refer to \cite{Biggins-branching-out} for a survey of this type of results.
Note that, as an immediate consequence,  $\zeta_- \le v^* \le \zeta_+$.

The following result shows that the existence of a $t^*$ such that  (\ref{e:lambda-1}) is satisfied (the $t^*$ assumption), is equivalent to the infimum defining $v^*$ in (\ref{e:def-de-v-astre}) being in fact a minimum. (Note that (\ref{e:tcondition}) below is exactly the same as (\ref{e:lambda-1}).)
\begin{lemma} \label{l:tmintcondition} 
For all $t^*>0$, 
\begin{equation} \label{e:tmin}
\frac{\theta(t^*)}{t^*}=v^*
\end{equation}
is equivalent to
\begin{equation} \label{e:tcondition}
\frac{\theta(t^*)}{t^*}=\theta'(t^*), \hbox{ i.e. } \Lambda(t^*)-t^* \Lambda'(t^*) = -\log(2).
\end{equation}
\end{lemma}
\begin{proof} 
Let $t^*>0$. If (\ref{e:tmin}) holds, then the map defined by $t \mapsto \theta(t)/t$ reaches its minimum at $t^*$.
But its derivative at $t^*$ is :
$$
\frac{\theta'(t^*)t^*-\theta(t^*)}{{t^*}^2}.
$$
Therefore (\ref{e:tcondition}) holds.
Conversely, assume that (\ref{e:tcondition}) holds.
By convexity of $\theta$ and by (\ref{e:tcondition}) we get, for all $t>0$,
$$
\theta(t) \ge \theta'(t^*)(t-t^*)+\theta(t^*)=\theta(t^*)t/t^*.
$$
Therefore (\ref{e:tmin}) holds. \end{proof}

The following lemma provide some alternative probabilistic interpretations for the $t^*$ assumption.
\begin{lemma}\label{l:tstar-gw} The following conditions are equivalent.
\begin{enumerate}
\item There exists $t^*>0$ such that (\ref{e:tcondition}) holds.
\item $P(\zeta > v^*)>0$, i.e.\ $v^*<\zeta_+$.
\item $P(\zeta=\zeta_+)<1/2$.
\end{enumerate}
\end{lemma}

\begin{proof}
By Lemma \ref{l:depasse-pas}, Condition 1 implies Condition 2. 
Let us check that Condition 2 implies condition 1.
We have:
$$
\lim_{t\to 0^+} \frac{\theta(t)}{t} = +\infty \hbox{ and } \lim_{t\to+\infty} \frac{\theta(t)}{t} = \zeta_+.
$$
Assume $v^*<\zeta_+$.
Then, $v^*$ can be defined as the infimum of $t \mapsto \theta(t)/t$ on a compact interval of $]0,+\infty[$.
Therefore, there exists $t^*>0$ such that (\ref{e:tmin}) holds.
By Lemma \ref{l:tmintcondition} we then get that Condition 1 holds. 

Let us check that Conditions 2 and 3 are equivalent.
This is a special case of Proposition A.2 in \cite{Jaffuel}.
Below we give a sligthly modified proof.
Let us write
$$
\theta(t)=t\zeta_++\log(A+B(t))
$$
where
$$
A=2P(\zeta=\zeta_+) \hbox{ and } B(t)=2E(\exp(t(\zeta-\zeta_+))1_{\zeta<\zeta_+}).
$$
Note that $A+B(t)$ converges toward $A$. 
If $P(\zeta=\zeta_+)<1/2$, we then get:
$$
\lim \log(A+B(t))<0
$$
(the limit might be $-\infty$).
Therefore there exists $t>0$ such that $\theta(t)/t<\zeta_+$ and we then have $v^*<\zeta_+$: Condition 2 holds.
If, on the contrary, $P(\zeta=\zeta_+) \ge 1/2$, then for all $t>0$ we have $A+B(t) \ge 1$ and then $\theta(t)/t \ge \zeta_+$.
Therefore $v^* \ge \zeta_+$: Condition 2 does not hold. \end{proof}

\subsection{Survival probability when the $t^*$ assumption is not fulfilled}

Our aim in this subsection is to show that, when the $t^*$ assumption is not fulfilled, the behavior of the survival probability differs from the one established in the core of the paper. 

Let us consider the process of children displaced by exactly $\zeta_+$ from their parent in the branching random walk. This is a Galton-Watson process. The number of children of 
a given parent follows a binomial distribution with parameters $2$ and $P(\zeta=\zeta_+)$.
The behavior of the surving probability $\P_0(A_{\infty}(v))$ as $v$ tends to $v^*$ from below 
turns out to depend on the behavior of this Galton-Watson process.
\begin{enumerate}
\item Subcritical case, i.e.\ $P(\zeta=\zeta_+)<1/2$. By Lemma \ref{l:tstar-gw} this is the case studied in the paper.
We have:
$$
\log \P_0(A_{\infty}(v)  ) \sim -C_{\zeta}(v^*-v)^{-1/2}
$$
where $C_{\zeta}$ is an explicit positive constant that depends on the distribution of $\zeta$.
\item Critical case, i.e.\ $P(\zeta=\zeta_+)=1/2$.
By Lemma \ref{l:tstar-gw} we have $v^*=\zeta_+$.
As in the subcritical case we have:
$$
\P_0(A_{\infty}(v)) \to \P_0(A_{\infty}(v^*)),
$$
since $\P_0(A_{\infty}(v^*))=0$ is also the survival probability of the critical Galton-Watson process.
However, the rate of convergence is not the same as in the subcritical case.
Indeed, there exists a positive constant $d_{\zeta}$ depending only on the distribution of $\zeta$ such that
$$
\P_0(A_{\infty}(v)) \ge d_{\zeta}(v^*-v).
$$
When $P(\zeta=0)=P(\zeta=1)=1/2$ this is the lower bound given by Proposition 2.4 of \cite{pem}.
The general case follows by coupling and rescaling 
\footnote{First replace $\zeta$ by a random variable $\widetilde{\zeta}$ such that $P(\widetilde{\zeta}=\zeta_+)=P(\widetilde{\zeta}=\zeta_-)=1/2$,
thus lowering the survival probability. Then replace $\widetilde{\zeta}$ by $(\widetilde{\zeta}-\zeta_-)/(\zeta_+-\zeta_-)$.}.

The precise behavior of the survival probability in fact depends on the behavior of $P(v \le \zeta < v^*)$.
When there exists $v<v^*$ such that $P(v \le \zeta <v^*)=0$ there exists a constant $D_{\zeta}$ such that 
$$
\P_0(A_{\infty}(v)) \le D_{\zeta}(v^*-v).
$$
When $P(\zeta=0)=P(\zeta=1)=1/2$ this is the upper bound given by Proposition 2.4 of \cite{pem}.
We get the previous result by coupling and rescaling as above.

When, on the contrary, $P(v \le \zeta <v^*)>0$ for all $v<v^*$ the survival probability can be larger.
Indeed, set $p_v=P(\zeta \ge v)>1/2$.
The Galton-Watson process obtained by keeping all steps greater or equal to $v$ is supercritical.
Denote by $q_v>0$ its survival probability.
We have
$$
(1-p_v+p_v(1-q_v))^2=(1-q_v).
$$
Therefore 
$$
q_v=\frac{2p_v-1}{p_v^2}.
$$
We then have:
$$
P_0(A_{\infty}(v)) \ge q_v = \frac{2p_v-1}{p_v^2} \sim 8P( v \le \zeta < v^*)
$$
(let us recall that $P(\zeta=v^*)=1/2$ and that $P(\zeta>v^*)=0$).
Thus, the survival probability $P_0(A_{\infty}(v))$ can tend arbitrarily slowly to $0$.

\item Supercritical case, i.e.\ $P(\zeta=\zeta_+)>1/2$.
By Lemma \ref{l:tstar-gw} we have $v^*=\zeta_+$, as in the critical case.
Let us denote by $q>0$ the survival probability of the Galton-Watson tree.
Then, we have:
$$
\P_0(A_{\infty}(v)) \to \P_0(A_{\infty}(v_*))=q>0.
$$
\end{enumerate}

\end{document}